\documentclass[
11pt,twoside, a4paper,
]{amsart}

\title[Brauer-Manin obstructions]
{           {\protect\hfill \normalfont \tiny
            \\ \vspace{10pt}}
Vanishing of  algebraic \\ Brauer-Manin obstructions
}
\author{Mikhail Borovoi}

\address{
Raymond and Beverly Sackler School of Mathematical Sciences,
Tel Aviv University, 69978 Tel Aviv, Israel}
\email{borovoi@post.tau.ac.il}

\thanks{Partially supported by the
Hermann Minkowski Center for Geometry
and by the Israel Science Foundation (grant 807/07)}


\usepackage[arrow,matrix]{xy}
\usepackage{mathptmx}
\usepackage{amscd}
\usepackage{amssymb}
\usepackage{eucal}
\usepackage{amsxtra}
\usepackage{verbatim}
\usepackage{url}
\DeclareFontEncoding{OT2}{}{} 
\DeclareTextFontCommand{\textcyr}{\fontencoding{OT2}
    \fontfamily{wncyr}\fontseries{m}\fontshape{n}\selectfont}
\newcommand{\Bcyr}{\textcyr{B}}

\usepackage{comment}

\theoremstyle{plain}

\newtheorem{theorem}{Theorem}[section]
\newtheorem{proposition}[theorem]{Proposition}
\newtheorem{lemma}[theorem]{Lemma}
\newtheorem{corollary}[theorem]{Corollary}

\newtheorem{conditional-result}[theorem]{Conditional Result}

\newtheorem{theorem?}{Theorem(?)}[section]
\newtheorem{proposition?}[theorem]{Proposition(?)}
\newtheorem{lemma?}[theorem]{Lemma(?)}
\newtheorem{corollary?}[theorem]{Corollary(?)}

\newtheorem*{theorem*}{Theorem}
\newtheorem*{proposition*}{Proposition}
\newtheorem*{lemma*}{Lemma}
\newtheorem*{corollary*}{Corollary}
\newtheorem*{question*}{Question}
\newtheorem*{conjecture*}{Conjecture}
\newtheorem*{claim*}{Claim}

\newtheorem*{introtheorem*}{Theorem}
\newtheorem*{introproposition*}{Proposition}
\newtheorem*{introlemma*}{Lemma}
\newtheorem*{introcorollary*}{Corollary}

\theoremstyle{definition}

\newtheorem{definition}[theorem]{Definition}
\newtheorem{example}[theorem]{Example}

\newtheorem*{definition*}{Definition}
\newtheorem*{example*}{Example}

\newtheorem{subsec}[theorem]{}

\newtheorem{question}[theorem]{Question}

\theoremstyle{remark}
\newtheorem{remark}[theorem]{Remark}

\newtheorem*{remark*}{Remark}

\DeclareSymbolFont{rsfs}{U}{rsfs}{m}{n}
\DeclareSymbolFontAlphabet{\mathcal}{rsfs}

\DeclareFontEncoding{OT2}{}{} 
\DeclareTextFontCommand{\textcyr}{\fontencoding{OT2}
    \fontfamily{wncyr}\fontseries{m}\fontshape{n}\selectfont}
\newcommand{\Sh}{\textcyr{Sh}}

\newcommand{\Be}{\textcyr{B}}

\newcommand{\sV}{{\mathcal{V}}}

\newcommand{\isoto}{\overset{\sim}{\to}}

\newcommand{\into}{\hookrightarrow}
\newcommand{\onto}{\twoheadrightarrow}

\newcommand{\labelto}[1]{\xrightarrow{\makebox[1.5em]{\scriptsize ${#1}$}}}

\renewcommand{\simeq}{\cong}

\def\uu{^\mathrm{u}}
\def\red{^\mathrm{red}}
\def\tor{^{\mathrm{tor}}}
\def\sc{^{\mathrm{sc}}}
\def\sss{^{\mathrm{ss}}}

\newcommand{\ZZ}{{\mathbb{Z}}}
\newcommand{\QQ}{{\mathbb{Q}}}

\newcommand{\Gal}{{\rm Gal}}

\newcommand{\Br}{{\rm Br}}
\renewcommand{\ker}{{\rm ker}}
\newcommand{\coker}{{\rm coker}}

\newcommand{\Hom}{{\rm Hom}}

\newcommand{\kbar}{{\overline{k}}}

\newcommand{\XX}{{\mathbb{X}^*}}

\newcommand{\pn}{\par\noindent}

\newcommand{\Gbar}{{\overline G}}
\newcommand{\Hbar}{{\overline H}}
\newcommand{\xbar}{{\overline x}}

\newcommand{\That}{{\widehat{T}}}
\newcommand{\Hhat}{{\widehat{H}}}

\newcommand{\Ghat}{{\widehat{G}}}

\def\Hbarhat{{\widehat{\overline{H}}}}

\newcommand{\Gm}{{\mathbb{G}_m}}
\renewcommand{\gg}{{\mathfrak{g}}}

\newcommand{\Aut}{{\rm Aut}}

\newcommand{\Xbar}{{\overline{X}}}

\newcommand{\Bro}{{\Br_1}}
\newcommand{\Bra}{{\Br_{{\rm a}}}}

\def\G{{\mathbb{G}}}
\def\Bhat{{\widehat{B}}}

\def\H{{\mathbb{H}}}

\def\xx{{\mathbb{X}_*}}

\def\tors{_{\mathrm{tors}}}
\def\psc{^{\prime\mathrm{\,sc}}}
\def\Ext{{\mathrm{Ext}}}
\def\Inn{\mathrm{Inn}}

\def\Btil{{\widetilde{B}}}
\def\psitil{{\widetilde{\psi}}}
\def\btil{{\widetilde{b}}}
\def\Ahat{{\widehat{A}}}

\begin{document}


\begin{abstract}
Let $X$ be a homogeneous space of a quasi-trivial $k$-group $G$,
with geometric stabilizer $\Hbar$, over a number field $k$.
We prove that under certain conditions on the character group of $\Hbar$, certain algebraic Brauer-Manin obstructions
to the Hasse principle and weak approximation vanish,
because the abelian groups where they take values vanish.
When $\Hbar$ is connected or abelian,
these algebraic Brauer-Manin obstructions to the Hasse principle and weak approximation are the only ones,
so we prove the Hasse principle and weak approximation for $X$ under certain conditions.
As an application, we obtain new sufficient conditions
for the Hasse principle and weak approximation for linear algebraic groups.
\end{abstract}

\maketitle
\tableofcontents

\section{Introduction: Sansuc's results}

We are inspired by Sansuc's paper \cite{Sansuc}.
In this section we state and discuss
Sansuc's results on the Hasse principle and weak approximation
for principal homogeneous spaces of connected linear algebraic groups admitting special coverings.

\begin{subsec}\label{subsec:spec-covering}
Let $k$ be a field of characteristic 0
and $\kbar$ be a fixed algebraic closure of $k$.
If $X$ is an algebraic variety over $k$, we write $\Xbar=X\times_k \kbar$.

A $k$-torus $T$ is called quasi-trivial if
its character group $\That:= \Hom(\overline{T}, \G_{m, \kbar})$
is a permutation $\Gal(\kbar/k)$-module, i.e. $\That$ has a $\ZZ$-basis which the Galois group permutes.

A {\em special covering} of a (connected) reductive $k$-group $G$
is a short exact sequence
$$
1\to B\to G'\to G\to 1,
$$
where $G'$ is a product of a simply connected semisimple $k$-group
and a quasi-trivial $k$-torus, and $B$ is a finite central subgroup of $G'$.

Not all reductive groups admit special coverings.
For example,  a nonsplit one-dimensional $k$-torus does not admit such a covering.

A finite group $\Gamma$ is called {\em metacyclic} if all its Sylow subgroups
are cyclic. Any cyclic group is metacyclic.
The group $\ZZ/2\ZZ\times \ZZ/2\ZZ$ is not metacyclic.
A  finite Galois extension $L/k$ is called metacyclic
if $\Gal(L/k)$ is a metacyclic group.
Clearly any cyclic extension is metacyclic.

Let $k$ be a field of characteristic 0  and $M$ be a $\Gal(\kbar/k)$-module.
We say that a field extension $K\subset \kbar$ of $k$ {\em splits} $M$
if $\Gal(\kbar/K)$ acts trivially on $M$.
If $T$ is a $k$-torus, then $K$ splits  $T$ if and only if $K$ splits $\That$.

Now let $k$ be a number field.
We denote by $\sV$ the set of all places $v$ of $k$,
and by $\sV_\infty$ the set of infinite (archimedean) places.
For $v\in \sV$ we denote by $k_v$ the completion of $k$ at $v$.

\end{subsec}

\begin{subsec} \label{subsec:Sansuc}
Let $G$ be a reductive $k$-group over a number field $k$ admitting a special covering
$$
1\to B\to G'\to G\to 1.
$$
Let $\Bhat$ denote the  character group of $B$,
i.e. $\Bhat=\Hom(B,\G_{m,\kbar})$.
Let $K$ be the smallest subfield of $\kbar$
splitting $\Bhat$
(i.e $K$ is the subfield corresponding to the subgroup
$\ker\left[\Gal(\kbar/k)\to\Aut(\Bhat)\right]$).
Let $X$ be a principal homogeneous space of $G$.

\end{subsec}

Sansuc \cite{Sansuc} proved the following results:

\begin{proposition}[\cite{Sansuc}, Cor.~5.2]
\label{prop:Sansuc-HP}
Let  $k,\ G,\ B,\ K$, and  $X$ be as in \ref{subsec:Sansuc}.
If $K/k$ is a metacyclic extension,
then $X$ satisfies the Hasse principle and weak approximation,
i.e. if $X(k_v)\neq\emptyset$ for any place $v$ of $k$,
then $X(k)\neq \emptyset$ and, moreover,
$X(k)$ is dense in $\prod_{v\in\sV} X(k_v)$.
\end{proposition}

\begin{remark}
If the extension $K/k$ is not metacyclic (e.g. $\Gal(K/k)\cong \ZZ/2\ZZ\oplus\ZZ/2\ZZ$),
then the Hasse principle or weak approximation may fail for $X$,
see \cite[\S III.4.7]{Serre} and \cite[Examples 5.6, 5.7, and 5.8]{Sansuc}.
\end{remark}

\begin{proposition}[\cite{Sansuc}, Thm.~3.3 and Lemma 1.6]
\label{prop:Sansuc-WA}
Let  $k,\ G,\ B,\ K$, and  $X$ be as in \ref{subsec:Sansuc}.
Assume that $X$ has a $k$-point (hence we may identify $X$ with $G$).
Let $S\subset \sV$ be a finite set formed by places of $k$
with cyclic decomposition groups in $K/k$
(for example, assume that $K/k$ is unramified at all finite places in $S$).
Then $X$ has the weak approximation property in $S$,
i.e. the set $X(k)$ is dense in $\prod_{v\in S}X(k_v)$.
\end{proposition}

\begin{subsec}
The results of Propositions \ref{prop:Sansuc-HP} and \ref{prop:Sansuc-WA}
can be explained in terms of the Brauer group of $X$.
Let $X$ be a smooth $k$-variety.
We write $\Br(X)$ for the cohomological Brauer group of $X$,
i.e $\Br(X)=H^2_{\text{\'et}}(X,\Gm)$.
We set $\Bro(X)=\ker[\Br(X)\to\Br(\Xbar)]$.
We define the \emph{algebraic Brauer group} $\Bra(X)$
by $\Bra(X)=\coker[\Br(k)\to\Bro(X)]$.

When $k$ is a number field and $S\subset \sV$ is a finite set of places of $k$,
we set
$$
\Be_S(X)=\ker\left[\Bra(X)\to\prod_{v\notin S} \Bra(X_{k_v})\right].
$$
Set  $\Be_\omega(X)=\bigcup_S \Be_S$.
We set $\Be(X):=\Be_\emptyset(X)$ and
 $\Be_{S,\emptyset}(X)=\Be_S(X)/\Be_\emptyset(X)=\Be_S(X)/\Be(X)$.

Sansuc \cite[(6.2.3)]{Sansuc}, following Manin \cite{Manin},  defined a natural pairing
$$
\prod_{v\in \sV}X(k_v)\times\Bcyr_\omega(X)\to \QQ/\ZZ
$$
(see also \cite[(5.2)]{Skorobogatov}),
which is continuous in the first argument and is additive in the second one.
This pairing induces a continuous map
\begin{equation}\label{eq:BM-HW}
m\colon \prod_{v\in \sV}X(k_v)\to \Hom(\Bcyr_\omega(X),\QQ/\ZZ).
\end{equation}
If $x_0\in X(k)\subset \prod_{v\in \sV}X(k_v)$, then $m(x_0)=0$.
If $m$ is not identically 0, say, $m(x_\sV)\neq 0$
for some   $x_\sV=(x_v)_{v\in \sV}\in\prod_{v\in \sV}X(k_v)$,
then $x_\sV$ is not contained in the closure of $X(k)$,
weak approximation fails for $X$.
We say that $m$ is {\em the algebraic Brauer-Manin obstruction
to the Hasse principle and weak approximation for $X$ associated with $\Be_\omega$}.

Assume that $X$ is a smooth $k$-variety having a $k$-point.
Let $S\subset \sV$ be a finite set of places of $k$.
Inspired by \cite{CTS} and \cite{Sansuc}, we defined in
 \cite[\S 1]{Bor96} a natural pairing
$$
\prod_{v\in S} X(k_v)\times \Bcyr_{S,\emptyset}(X)\to \QQ/\ZZ.
$$
This pairing induces a continuous map
\begin{equation}\label{eq:BM-W}
m_S\colon \prod_{v\in S}X(k_v)\to\Hom(\Bcyr_{S,\emptyset}(X),\QQ/\ZZ).
\end{equation}
If $x_0\in X(k)\subset \prod_{v\in S} X(k_v)$, then $m_S(x_0)=0$.
If $m_S$ is not identically 0, say $m_S(x_S)\neq 0$ for some
$x_S\in  \prod_{v\in S} X(k_v)$, then $x_S$ is not contained
in the closure of $X(k)$, hence weak approximation in $S$ fails for $X$.
We say that $m_S$ is {\em the algebraic Brauer-Manin obstruction
to  weak approximation in $S$ for $X$ associated with $\Bcyr_{S,\emptyset}(X)$}.

Using Sansuc's methods and results, one can show
that under the assumptions of Proposition \ref{prop:Sansuc-HP}
we have $\Be_\omega=0$, hence $m=0$, see Proposition \ref{prop:metacyclic}(ii) below.
Moreover, the Hasse principle and weak approximation
hold for $X$ because there is no Brauer-Manin obstruction.
Similarly, under the assumptions of Proposition \ref{prop:Sansuc-WA}
we have $\Be_{S,\emptyset}=0$, hence $m_S=0$,
see Proposition \ref{prop:metacyclic}(i) below.
Again, weak approximation in $S$ holds for $X$
because there is no Brauer-Manin obstruction.
We provide some details.
\end{subsec}

\begin{proposition}\label{prop:metacyclic}
Let  $k,\ G,\ B,\ K$, and  $X$ be as in \ref{subsec:Sansuc}.
Let $S\subset \sV$ be a finite set of places of $k$.
\par(i)  If  any place $v\in S$ has a cyclic decomposition group in $K/k$,
then $\Be_{S,\emptyset}(X)=0$.
\par(ii) If  $K/k$ is a metacyclic extension, then $\Be_\omega(X)=0$.
\end{proposition}

\begin{proof}
We use the notation of \S \ref{subsec:Sh} below.
By \cite[Lemma 6.8]{Sansuc} $\Bra(X)\cong\Bra(G)$,
hence $\Be_\omega(X)\cong\Be_\omega(G)$ and $\Be_{S,\emptyset}(X)\cong \Be_{S,\emptyset}(G)$.
By \cite[Prop. 9.8]{Sansuc} $\Be_\omega(G)\cong\Sh^1_\omega(k,\Bhat)$ and $\Be(G)\cong\Sh^1(k,\Bhat)$.
One proves similarly that $\Be_S(G)\cong\Sh^1_S(k,\Bhat)$, hence
$\Be_{S,\emptyset}(G)\cong\Sh^1_{S,\emptyset}(k,\Bhat):=\Sh^1_S(k,\Bhat)/\Sh^1_\emptyset(k,\Bhat)$.
In case (i), since $S$ is formed by places of $k$ with cyclic decomposition groups in $K/k$,
by \cite[Lemma 1.1(iii)]{Sansuc} (see also Lemma \ref{lem:S-1.1}(iii) below)  $\Sh^1_{S,\emptyset}(k,\Bhat)=0$,
hence $\Be_{S,\emptyset}(X)=0$.
In case (ii), since $K/k$ is metacyclic,
by \cite[Lemma 1.3]{Sansuc} (see also Lemma \ref{lem:S-1.3} below) $\Sh^1_\omega(k,\Bhat)=0$,
hence $\Be_\omega(X)=0$.
\end{proof}

\begin{proof}[An alternative proof of Proposition \ref{prop:Sansuc-HP}.]
By Proposition \ref{prop:metacyclic}(ii)
the Brauer-Manin obstruction $m$ of formula \eqref{eq:BM-HW} is identically zero in our case,
i.e there is no algebraic Brauer-Manin obstruction to the Hasse principle and weak approximation
associated with $\Be_\omega$.
Since by \cite[Cor.~8.7 and Cor.~8.13]{Sansuc} this obstruction is the only one,
we conclude that the Hasse principle and weak approximation hold for $X$.
\end{proof}

\begin{proof}[An alternative proof of Proposition \ref{prop:Sansuc-WA}.]
 By Proposition \ref{prop:metacyclic}(i)
the Brauer-Manin obstruction $m_S$ of formula \eqref{eq:BM-W} is identically zero in our case,
i.e. there is no algebraic Brauer-Manin obstruction to weak approximation in $S$
associated with $\Be_{S,\emptyset}(X)$.
Since by \cite[Cor. 8.13]{Sansuc} this obstruction is the only one,
we conclude that  weak approximation in $S$ holds for $X$.
\end{proof}

\section{Introduction (continued): our main results}

In this section we state our  generalizations of Sansuc's results in two cases:
homogeneous spaces of quasi-trivial groups
and principal homogeneous spaces of connected linear algebraic groups.
Our main results are Theorems
\ref{thm:WA-S-H} and \ref{thm:WA-S-pi1},
generalizing Proposition \ref{prop:metacyclic}
and proving that the groups  $\Be_{S,\emptyset}(X)$ and $\Be_\omega(X)$ vanish under certain conditions on $X$.

In order to state our results  we use the notion of a quasi-trivial group,
introduced by Colliot-Th\'el\`ene \cite[Definition 2.1]{CT}, see also Definition \ref{def:qt} below.

Let $X$ be a homogeneous space of a  quasi-trivial $k$-group $G$ over a number field $k$.
Let $\Hbar\subset\Gbar$ be the stabilizer of a $\kbar$-point $\xbar\in X(\kbar)$
(we do not assume that $\Hbar$ is connected or abelian).
It is well known that the character group $\Hbarhat$ of $\Hbar$ has a canonical structure of a Galois module,
see \cite[4.1]{Bor96}  or \cite[Rem.~5.7(1)]{BvH2}, see also  \S \ref{subsec:Galois-action} below.

\begin{theorem}\label{thm:WA-S-H}
Let $X$ be a left homogeneous space of a  quasi-trivial $k$-group $G$ over a number field $k$.
Let $\Hbar\subset\Gbar$ be the stabilizer of a $\kbar$-point $\xbar\in X(\kbar)$
Let $S\subset \sV$ be a finite subset.
Let $K/k$ be the smallest Galois extension in $\kbar$ splitting the Galois module $\Hbarhat$.
\par(i)  If  any place $v\in S$ has a cyclic decomposition group in $K/k$,
then $\Be_{S,\emptyset}(X)=0$.
\par(ii) If  $K/k$ is a metacyclic extension,
then $\Be_\omega(X)=0$.
\end{theorem}

Theorem \ref{thm:WA-S-H} will be proved in Section \ref{sec:HS}.

\begin{corollary}\label{cor:WA}
Let $X$ be as in Theorem \ref{thm:WA-S-H}.
Assume that $X$ has a $k$-point,
i.e. $X=G/H$, where $G$ is a quasi-trivial $k$-group and $H\subset G$ is a $k$-subgroup.
Assume that $H$ is connected or abelian.
Let $S\subset\sV$ be a finite set of places of $k$
formed by places with cyclic decomposition groups in $K/k$
(for example assume that $K/k$ is unramified at all finite places in $S$).
Then $X$ has weak approximation in $S$.
\end{corollary}

\begin{proof}
Since $H$ is connected or abelian, the algebraic Brauer-Manin obstruction
$m_S$ associated with $\Be_{S,\emptyset}$ is the only
obstruction to  weak approximation in $S$ for $X$,
see \cite[Thm.~2.3]{Bor96}.
By Theorem \ref{thm:WA-S-H}(i) $m_S=0$, hence $X$ has weak approximation in $S$.
\end{proof}

\begin{corollary}\label{cor:RA}
Let $X$ be as in Corollary \ref{cor:WA},
i.e. $X=G/H$, where $G$ is a quasi-trivial $k$-group and $H\subset G$ is a $k$-subgroup.
Assume that $H$ is connected or abelian.
Then $X$ has the real approximation property,
i.e $X(k)$ is dense in $\prod_{v\in \sV_\infty}X(k_v)$.
\end{corollary}

\begin{proof}
For any $v\in\sV_\infty$ the decomposition group of $v$ in $K/k$ is cyclic,
and by Corollary \ref{cor:WA} $X$ has weak approximation in $\sV_\infty$,
i.e. real approximation.
\end{proof}

\begin{question}\label{q:hom-spaces-real-approximation}
Does there exist a homogeneous space $X=G/H$, where $G$ is a quasi-trivial $k$-group over a number field $k$,
and $H\subset G$ is a nonconnected non-abelian  $k$-subgroup,
such that real approximation fails for $X$?
\end{question}

\begin{corollary}\label{cor:RA-connected stabilizer}
Let $X$ be a homogeneous space having a $k$-rational point with connected stabilizer,
of a connected linear algebraic group (not necessarily quasi-trivial)
over a number field $k$;
in other words, $X= G/H$, where $G$ is a connected $k$-group
and $H\subset G$ is a connected $k$-subgroup.
Then $X$ has the real approximation property.
\end{corollary}

\begin{proof}
By Lemma \ref{lem:repr-quasi-trivial} below we can write
$X=G'/H'$, where $G'$ is a quasi-trivial $k$-group
and $H'\subset G'$ is a connected $k$-subgroup.
Now by Corollary \ref{cor:RA} $X$ has real approximation.
\end{proof}

\begin{corollary}\label{cor:HP}
Let $X$ be as in Theorem \ref{thm:WA-S-H}.
Assume that $\Hbar$ is connected or abelian.
Assume that $K/k$ is a metacyclic extension.
Then $X$ satisfies the Hasse principle and weak approximation.
\end{corollary}

\begin{proof}
Since $\Hbar$ is connected or abelian,
the algebraic Brauer-Manin obstruction $m$ associated with $\Be_\omega$ is the only
obstruction to the Hasse principle and weak approximation,
see \cite[Thms. 2.2 and 2.3]{Bor96}.
By Theorem \ref{thm:WA-S-H}(ii) we have $m=0$, hence $X$ satisfies the Hasse principle and weak approximation.
\end{proof}

Note that Propositions \ref{prop:Sansuc-HP} and \ref{prop:Sansuc-WA} (due to Sansuc)
follow immediately from our Corollaries \ref{cor:HP} and \ref{cor:WA}, respectively.
Note also that the special case of Corollary \ref{cor:WA}
when $G$ is simply connected and $H$ is connected was earlier proved in \cite[Cor.~1.6]{Bor90} by a different method.
\medskip

In order to state our results on principal homogeneous spaces of connected $k$-groups,
we use the notion of the \emph{algebraic fundamental group} $\pi_1(G)$ introduced in \cite[\S1]{Bor98}
(where we wrote $\pi_1(\Gbar)$ instead of $\pi_1(G)$),
see also \cite[\S 6]{CT}.
Note that $\pi_1(G)$ a finitely generated Galois module.

\begin{theorem}\label{thm:WA-S-pi1}
Let $G$ be a connected linear $k$-group over a number field $k$.
Let $X$ be a  principal homogeneous space (torsor) of $G$ over $k$.
Let $S\subset \sV$ be a finite set of places of $k$.
Let $K/k$ be the smallest Galois extension in $\kbar$ splitting the Galois module $\pi_1(G)$.
\par(i)  If  any place $v\in S$ has a cyclic decomposition group in $K/k$,
then $\Be_{S,\emptyset}(X)=0$.
\par(ii) If  $K/k$ is a metacyclic extension,
then $\Be_\omega(X)=0$.
\end{theorem}

Theorem \ref{thm:WA-S-pi1} will be proved in Section \ref{sec:PHS}.

\begin{proposition}\label{thm:extension}
Let $G$ be a connected linear algebraic group over a field $k$ of characteristic 0.
Let $K/k$ be the smallest Galois extension in $\kbar$ splitting $\pi_1(G)$.
Then there exists an exact sequence
$$
1\to H\to G'\to G\to 1,
$$
where $G'$ is a quasi-trivial $k$-group and $H$ is a central $k$-subgroup of multiplicative type,
such that $K$ splits both $\Hhat$ and $\widehat{G'}$.
\end{proposition}

Proposition \ref{thm:extension} will be proved in Section \ref{sec:extension}.
In Section \ref{sec:alternative} we shall give an alternative proof of Theorem \ref{thm:WA-S-pi1}
based on Proposition \ref{thm:extension} and Theorem \ref{thm:WA-S-H}.

\begin{corollary}\label{cor:WA-pi1}
Let $k$, $G$, and $K$ be as in Theorem \ref{thm:WA-S-pi1}.
Let $S\subset\sV$ be a finite set of places of $k$
formed by places with cyclic decomposition groups in $K/k$
(for example assume that $K/k$ is unramified at all finite places in $S$).
Then $G$ has weak approximation in $S$.
\end{corollary}

\begin{proof}
Under our assumptions the algebraic Brauer-Manin obstruction $m_S$
associated with $\Be_{S,\emptyset}(X)$  is the only
obstruction to  weak approximation in $S$ for $G$,
 see \cite[Cor.~8.13]{Sansuc}.
By Theorem \ref{thm:WA-S-pi1}(i) $m_S=0$, hence $G$ has weak approximation in $S$.
Alternatively, the corollary follows from Proposition \ref{thm:extension} and Corollary \ref{cor:WA}.
\end{proof}

\begin{corollary}\label{cor:HP-pi1}
Let $k$, $G$, $X$, and $K$ be as in Theorem \ref{thm:WA-S-pi1}.
Assume that $K/k$ is a metacyclic extension.
Then $X$ satisfies the Hasse principle and weak approximation.
\end{corollary}

\begin{proof}
Under our assumptions the algebraic Brauer-Manin obstruction $m$
associated with $\Be_\omega$ is the only
obstruction to the Hasse principle and weak approximation for $X$,
see \cite[Cor.~8.7 and Cor.~8.13]{Sansuc}.
By Theorem \ref{thm:WA-S-pi1}(ii) we have $m=0$,
hence $X$ satisfies the Hasse principle and weak approximation.
Alternatively, the corollary follows from Proposition \ref{thm:extension} and Corollary \ref{cor:HP}
(because $K$ splits $\Hbarhat$, where $\Hbar$ is the stabilizer of $\xbar$ in $\overline{G'}$,
see \S \ref{subsec:alternative} below).
\end{proof}

\begin{remark}
Sansuc  proved in \cite[Cor.~3.5(iii)]{Sansuc} that any connected $k$-group over a number field $k$ has the real approximation property.
This follows from our Corollary \ref{cor:WA-pi1} (because infinite places have cyclic decomposition groups in $\Gal(K/k)$)
and from our Corollary \ref{cor:RA-connected stabilizer}
(because we may write $G=G/\{1\}$, and $\{1\}$ is a connected $k$-subgroup).
\end{remark}

Note that Sansuc proved the following result:

\begin{proposition}[Sansuc] \label{prop:Sansuc-splitting}
Let $G$ be a connected linear $k$-group over a number field $k$.
Let $S\subset \sV$ be a finite subset.
Assume that $G$ splits over a finite Galois extension $K/k$.
\par(i)  If  any place $v\in S$ has a cyclic decomposition group in $K/k$,
then $G$ has weak approximation in $S$ (cf. \cite[Cor.~3.5(ii)]{Sansuc}).
\par(ii) If  $K/k$ is a metacyclic extension,
then $\Be_\omega(X)=0$ (cf. \cite[Prop.~9.8]{Sansuc}),
hence any principal homogeneous space $X$ of $G$ over $k$
satisfies the Hasse principle and weak approximation.
\end{proposition}

Proposition \ref{prop:Sansuc-splitting} follows from our Theorem \ref{thm:WA-S-pi1}:
if a finite Galois extension $K/k$ splits $G$, then it splits $\pi_1(G)$.
The following example shows that  Theorem \ref{thm:WA-S-pi1}
is indeed stronger than Proposition \ref{prop:Sansuc-splitting}.

\begin{example}
Let $k$ be a number field, and let $K_1$ and $K_2$ be two different quadratic extensions of $k$ in $\kbar$.
Let $K$ be the composite of $K_1$ and $K_2$, then $K/k$ is a Galois extension with non-metacyclic Galois group
$\Gal(K/k)\cong\ZZ/2\ZZ\times\ZZ/2\ZZ$.
Set $G_1=\text{SU}_{2,K_1}$, it is a $k$-group.
Set
$$
G_2=R^1_{K_2/k}\Gm:=\ker[N_{K_2/k}\colon R_{K_2/k}\G_{m,K_2}\to\G_{m,k}],
$$
where $N_{K_2/k}$ is the norm map.
Set $\mu=\{(-1,-1),(1,1)\}\subset G_1\times G_2$, and set $G=(G_1\times G_2)/\mu$.

Clearly $G$ does not admit a special covering.
Let $L/k$ be any finite Galois extension in $\kbar$ splitting $G$.
Then $L$ splits both $G_1$ and $G_2$, hence $L\supset K$, and therefore $L/k$ is not metacyclic.
We see that we cannot prove the Hasse principle and weak approximation for
a principal homogeneous space $X$ of $G$
using Proposition \ref{prop:Sansuc-HP} or Proposition \ref{prop:Sansuc-splitting}.

On the other hand, $\pi_1(G)\cong\ZZ$ as an abelian group, and $\Gal(\kbar/k)$ acts on $\pi_1(G)$ via $\Gal(K_2/k)$.
We see that the cyclic Galois extension $K_2/k$ splits $\pi_1(G)$.
By Theorem \ref{thm:WA-S-pi1}(ii) $\Be_\omega(X)=0$, and by Corollary \ref{cor:HP-pi1}
$X$ satisfies the Hasse principle and weak approximation.
\end{example}

The plan of the rest of this paper is as follows.
In Sections \ref{sec:finitely-generated}  we give preliminaries
on Galois cohomology of finitely generated Galois modules.
In Section \ref{sec:quasi-trivial} we give preliminaries
on quasi-trivial groups and on the algebraic fundamental group of a connected linear algebraic group.
In Sections  \ref{sec:HS} and  \ref{sec:PHS} we prove Theorems
\ref{thm:WA-S-H} and \ref{thm:WA-S-pi1}, respectively.
In Section  \ref{sec:extension} we prove Proposition \ref{thm:extension}, and in Section \ref{sec:alternative}
we use this proposition in order to give an alternative proof of Theorem \ref{thm:WA-S-pi1}.
Our proofs are based on the results of Section  \ref{sec:finitely-generated}
and of our papers \cite{BvH} and \cite{BvH2}.


\section{Preliminaries on  Galois cohomology of finitely generated Galois modules}\label{sec:finitely-generated}

\begin{subsec}\label{subsec:Sh}
In this section $k$ denotes a number field, and $B$ is a discrete $\Gal(\kbar/k)$-module
which is finitely generated as an abelian group (we say just ``a finitely generated Galois module'').
By $S$ we always denote a finite subset of $\sV$.
We write
$$
\Sh^i_S(k,B)=\ker\left[ H^i(k, B)\to \prod_{v\notin S} H^i(k_v,B)\right].
$$
We have $\Sh^i_\emptyset(k,B)=\Sh^i(k,B)$.
We set $\Sh^i_{S,\emptyset}(k,B):=\Sh^i_S(k,B)/\Sh^i_\emptyset(k,B)$ and
 $\Sh^i_\omega(k,B)=\bigcup_S\Sh^i_S(k,B)$.
\end{subsec}

Lemmas \ref{lem:S-1.1} and \ref{lem:S-1.3} below
are immediate generalizations of Sansuc's lemmas from \cite[\S 1a]{Sansuc}.
The only difference is that
Sansuc assumes that $B$ is a \emph{finite} Galois module, while we assume that $B$ is a finitely generated
Galois module.
Since Sansuc's proofs are concise, we give detailed proofs here.

\begin{lemma}\label{lem:S-1.1}
Let $K/k$ be a finite Galois extension with Galois group $\gg$ and $S$ be a finite set of places of $k$.

\par (i) If $B$ is a constant $\Gal(\kbar/k)$-module (i.e. $\Gal(\kbar/k)$ acts trivially),
then $\Sh^1_S(K/k,B)=0$ and $\Sh^1_S(k,B)=0$.
\par (ii) If the extension $K/k$ trivializes $B$, there is a reduction
$$
\Sh^1_S(k,B)=\Sh^1_S(K/k,B).
$$
\par (iii) If $S'$ is a finite subset of $\sV$ formed of places with cyclic decomposition groups in $K/k$,
then
$$
\Sh^1_{S\cup S'}(K/k,B)=\Sh^1_{S}(K/k,B).
$$
\end{lemma}

\newcommand{\Inf}{{\rm Inf\,}}

\begin{proof}
We follow closely \cite[Proof of Lemma 1.1]{Sansuc}.

We prove (i). We write $\gg=\Gal(K/k)$.
Since by assumption the finite group $\gg$ acts on $B$ trivially,
the group $\Sh^1_S(K/k,B)$ is the kernel of the homomorphism
\begin{equation}\label{e:Hom-gr}
\Hom_{\rm gr}(\gg,B)\,\longrightarrow \, \prod_{v\notin S}\,\Hom_{\rm gr}(\gg^v,B).
\end{equation}
where $\Hom_{\rm gr}$ denotes the group of group homomorphisms, and
$\gg^v$ denotes the decomposition group of a place $w$  of $K$ prolonging $v$.
By Chebotarev's density theorem, any element $\gamma$ of $\gg$ is conjugate to an element of some $\gg^v$
with $v\notin S$ (because the set $S$ is finite).
It follows that the homomorphism \eqref{e:Hom-gr} is injective, and hence, $\Sh^1_S(K/k,B)=0$.

Now let $c\in H^1(k,B)=\Hom_{\rm c.gr}(\Gal(\kbar/k), B)$,
where $\Hom_{\rm c.gr}$ denotes the group of continuous homomorphisms of groups.
Since $B$ is discrete and $\Gal(\kbar/k)$ is profinite,
we see that $\ker\, c$ is an open normal subgroup of $\Gal(\kbar/k)$,
and we may write $\ker\, c=\Gal(\kbar/K)$,
where $K$ is a finite Galois extension of $k$ contained in $\kbar$.
Let $c_K$ denote the natural homomorphism $\Gal(K/k)\to B$;
then $c$ is the inflation $\Inf c_K$.
Let $v\in \sV\smallsetminus S$ and let $w$ be a prolongation of $v$ to $K$.
Let $c_{K,v}$ denote the image of $c_K$ in $H^1(K_w/k_v, B)=\Hom(\Gal(K_w/k_v), B)$
and $c_v$ denote the image of $c$ in $H^1(k_v, B)=\Hom(\Gal(\kbar_v/k_v), B)$.
By assumption $c_v=0$.
Since $c_v=\Inf c_{K,v}$ and $\Gal(\kbar_v/k_v)$ acts on $B$ trivially,  we see that $c_{K,v}=0$.
Thus $c_K\in \Sh^1_S(K/k,B)$.
Since $\Sh^1_S(K/k,B)=0$, see above, we conclude that $c_K=0$, and hence, $c=0$.
Thus $\Sh^1_S(k,B)=0$, as required.

We prove (ii).
We assume that the extension $K/k$ trivializes $B$.
Let $S'$ denote the set of the prolongations to $K$ of the places in $S$.
Then by (i) we have  $\Sh^1_{S'}(K,B)=0$.
We consider the following commutative diagram with exact rows:
\[
\xymatrix{
0\ar[r] & H^1(K/k,B)\ar[r]\ar[d]_\alpha                       & H^1(k,B)\ar[r]\ar[d]_\beta               & H^1(K,B)\ar[d]^{\varphi} \\
0\ar[r] & \prod_{v\notin S} H^1(K_{w}/k_v,B)\ar[r]\  & \prod_{v\notin S} H^1(k_v,B)\ar[r] & \prod_{w\notin S'} H^1(K_w,B)
}
\]
obtained from inflation-restriction exact sequences,
after choosing, for each place $v$ of $k$, a prolongation $w$ of $v$ from $k$ to $K$.
Since $\Sh^1_{S'}(K,B)=0$, the homomorphism $\varphi$ in the diagram is injective.
It follows that $\Sh^1_S(K/k,B)=\ker\,\alpha\cong\ker\,\beta=\Sh^1_S(k,B)$, as required.

Chebotarev's density theorem implies that for each place $v$
with cyclic decompositions group $\gg^v$ (defined up to conjugacy in $\gg$),
there exists infinitely many other places $v'$ with the same decomposition group (again, up to conjugacy),
which proves (iii).
\end{proof}

Consider the homomorphism $\Gal(\kbar/k)\to\Aut(B)$.
The image $\gg$ of this homomorphism is finite.
Let $K$ denote the subfield in
$\kbar$ corresponding to the kernel of this homomorphism, then $K/k$
is a finite Galois extension with Galois group $\gg$.
We say that $K/k$ be the smallest Galois extension in $\kbar$ splitting $B$.

\begin{corollary}\label{cor:4cases}
Let $K/k$ be the smallest Galois extension in $\kbar$ splitting $B$,
and let $S\subset \sV$ be a finite set of places of $k$.
If  all places $v\in S$ have cyclic decomposition groups in $\Gal(K/k)$,
then  $\Sh^1_{S,\emptyset}(k,B)=0$.
\qed
\end{corollary}

\begin{proof}
By Lemma \ref{lem:S-1.1}(ii), $\Sh^1_S(k,B)=\Sh^1_S(K/k,B)$
and $\Sh^1_\emptyset(k,B)=\Sh^1_\emptyset(K/k,B)$.
By Lemma \ref{lem:S-1.1}(iii) $\Sh^1_S(K/k,B)=\Sh^1_\emptyset(K/k,B)$.
Thus $\Sh^1_S(k,B)=\Sh^1_\emptyset(k,B)$ and $\Sh^1_{S,\emptyset}(k,B)=0$.
\end{proof}

Recall that the exponent of a finite group is the least common multiple of the orders of its elements.
A finite group is metacyclic if and only if its exponent is equal to its order.

\begin{lemma}\label{lem:S-1.3}
Let $K/k$ be the smallest Galois extension in $\kbar$ splitting $B$, and
let $n$ and $e$ be the order and the exponent of  $\gg=\Gal(K/k)$, respectively.
Then multiplication by $n/e$ equals $0$ in $\Sh^1_\omega(k,B)$.
In particular, if $K/k$ is a metacyclic extension, then
$$
\Sh^1_\omega(k,B)=0,\text{ and hence, }\Sh^1(k,B)=0.
$$
The same holds if $B$ is finite and its order is relatively prime to $n/e$.
\qed
\end{lemma}

\begin{proof}
We follow closely \cite[Proofs of Lemmas 1.2 and  1.3]{Sansuc}.

First, we show that the group $\Sh^1_\omega(k,B)$ is a finite group, namely,
\begin{equation}\label{e:Sansuc-1.2.1}
\Sh^1_\omega(k,B)=\ker\left[H^1(\gg,B)\to\prod_{g\in\gg} H^1(\langle g \rangle,B)\right],
\end{equation}
where $\gg$ denotes the Galois group of the smallest Galois extension $K/k$ in $\kbar$ splitting $B$,
and $\langle g\rangle$ denotes the subgroup of $\gg$ generated by $g$.
We can  write
\begin{equation}\label{e:Sansuc-1.2.2}
\Sh^1_\omega(k,B)=\Sh^1_{S_0}(K/k,B),
\end{equation}
where $S_0$ denotes the finite subset of $\sV$ consisting of the places (ramified in $K$)
whose decomposition groups in $K/k$ are noncyclic.

Indeed, it follows from the definition of $K$ and  Lemma \ref{lem:S-1.1}(ii), that $\Sh^1_S(k,B)=\Sh^1_S(K/k,B)$.
Further, by Lemma \ref{lem:S-1.1}(iii), $\Sh^1_S(K/k,B)=\Sh^1_{S\cap S_0}(K/k,B)$.
This proves that the group $\Sh^1_\omega(k,B)$ is finite, and gives \eqref{e:Sansuc-1.2.2}.
The formula \eqref{e:Sansuc-1.2.1} follows from \eqref{e:Sansuc-1.2.2}, because by the construction of $S_0$,
for any $v\notin S_0$ we have $\gg^v=\langle g\rangle$ for some $g\in \gg$, and by Chebotarev's density theorem
any cyclic subgroup $\langle g\rangle\subseteq \gg$ is conjugate to a decomposition group $\gg^v$.

Now let $g$ be an element of $\gg$, and let $n_g$ denote its order.
The map
\[{\rm Cor}\circ{\rm Res}\colon H^1(\gg,B)\to H^1(\langle g\rangle, B)\to  H^1(\gg,B)\]
is multiplication by $n/n_g$; see \cite[Section 6, Proposition 8]{AW}.
By  \eqref{e:Sansuc-1.2.1}, the restriction of this map to  $\Sh^1_\omega(k,B)$ is 0,
which proves the first assertion of the lemma, from which the other assertions follow immediately.
 \end{proof}


\section{Preliminaries on quasi-trivial groups}\label{sec:quasi-trivial}

\begin{subsec}\label{subsec:connected-groups}
Let $k$ be a field of characteristic 0, $\kbar$ a fixed algebraic closure of $k$.
Let $G$ be a connected linear $k$-group.
We set $\Gbar=G\times_k \kbar$.
We use the following notation:
\pn $G\uu$ is the unipotent radical of $G$;
\pn $G\red=G/G\uu$ (it is reductive);
\pn $G\sss$ is the derived group of $G\red$ (it is semisimple);
\pn $G\sc$ is the universal cover of $G\sss$ (it is simply connected);
\pn $G\tor=G\red/G\sss$ (it is a torus);
\end{subsec}

\begin{definition}[ \cite{CT}, Prop. 2.2]\label{def:qt}
A connected linear $k$-group $G$ over a field $k$ of characteristic 0
is called \emph{quasi-trivial},
if $G\tor$ is a quasi-trivial torus and $G\sss$ is simply connected.
\end{definition}

\begin{lemma}[well known]
\label{lem:repr-quasi-trivial}
Let $k$ be a field of characteristic 0 and let $X$ be
a left homogeneous space of a connected linear $k$-group $G$.
Let $\overline{H}\subset \overline{G}$ be the stabilizer of a point $\xbar\in X(\kbar)$;
we assume that $\overline{H}$ is connected.
Then the variety $X$ is a homogeneous space of some quasi-trivial $k$-group $G'$
such that the stabilizer $\overline{H'}\subset\overline{G'}$ of $\xbar$ in $\overline{G'}$ is connected.
\end{lemma}

\begin{proof}
The lemma follows e.g. from \cite[Prop.-Def. 3.1]{CT}.
\end{proof}

\section{Homogeneous spaces of quasi-trivial groups}\label{sec:HS}

In this section we prove Theorem \ref{thm:WA-S-H}.

\begin{subsec}\label{subsec:complexes}
Let $k$ be a field of characteristic 0,
and let $A\to B$ be a morphism of $\Gal(\kbar/k)$-modules.
We write $\H^i(k,A\to B)$ for the Galois hypercohomology of the complex $A\to B$,
where $A$ is in degree 0 and $B$ is in degree 1.

When $k$ is a number field, we define  $\Sh^i_S(k,A\to B)$,  $\Sh^i_{S,\emptyset}(k,A\to B)$,
and $\Sh^i_\omega(k,A\to B)$  as in \S \ref{sec:finitely-generated}.
\end{subsec}

The following lemma must be well known (see \cite[Proof of Lemma 4.4]{Bor96} and  \cite[Proof of Cor.~2.15]{BvH} for similar results)
but we do not know a reference where it was stated, so we state and prove it here.

\begin{lemma}\label{lem:Sha}
 Let $k$ be a number field and $P\to L$ a complex of $\Gal(\kbar/k)$-modules in degrees 0 and 1,
where $P$ is a permutation $\Gal(\kbar/k)$-module.
Then for any finite set $S$ of places of $k$ we have a canonical isomorphism $\Sh_S^1(k,L)\isoto\Sh_S^2(k,P\to L)$.
\end{lemma}

\begin{proof}
We have an exact sequence
$$
0=H^1(k,P)\to H^1(k,L)\to \H^2(k, P \to L)\to H^2(k,P),
$$
and similar exact sequences for Galois cohomology over $k_v$ for $v\notin S$.
We obtain a commutative diagram with exact rows
$$
\xymatrix{
0\ar[r] &H^1(k,L)\ar[r]\ar[d]        &\H^2(k,P\to L)                   \ar[r]\ar[d] &H^2(k,P)\ar[d]  \\
0\ar[r] &\prod_{v\notin S}H^1(k_v,L)\ar[r] &\prod_{v\notin S}\H^2(k_v,P\to L)\ar[r] &\prod_{v\notin S}H^2(k_v,P).
}
$$
Since $P$ is a permutation module, we have $\Sh^2_S(k,P)=0$, cf. \cite[(1.9.1)]{Sansuc}.
An easy diagram chase shows that the homomorphism $\Sh^1_S(k,L)\to\Sh^2_S(k,P\to L)$
induced by this diagram is an isomorphism.
\end{proof}

\begin{proposition}\label{prop:Sh1}
Let $X$ be a homogeneous space of a  quasi-trivial $k$-group $G$ over a number field $k$.
Let $\Hbar\subset\Gbar$ be the stabilizer of a $\kbar$-point $\xbar\in X(\kbar)$.
Let $S\subset \sV$ be a finite set of places of $k$.
Then there is a canonical isomorphism $\Be_S(X)\simeq\Sh^1_S(k,\Hbarhat)$.
\end{proposition}

\begin{proof}
By \cite[Thm.~7.2]{BvH2} we have a canonical, functorial in $k$ isomorphism
$$
\Bra(X)\isoto \H^2(k, \Ghat\to\Hbarhat),
$$
whence we obtain a canonical isomorphism
$$
\Be_S(X)\simeq\Sh^2_S(k,\Ghat\to \Hbarhat).
$$
Since $\Ghat$ is a permutation module, by Lemma \ref{lem:Sha} we have a canonical isomorphism
$$
\Sh^1_S(k,\Hbarhat)\isoto\Sh^2_S(k,\Ghat\to \Hbarhat).
$$
Thus we obtain a canonical isomorphism $\Be_S(X)\cong\Sh^1_S(k,\Hbarhat)$.
\end{proof}

\begin{subsec}{\em Proof of Theorem \ref{thm:WA-S-H}.}
By Proposition \ref{prop:Sh1}
we have a canonical isomorphism $\Be_S(X)\simeq\Sh^1_S(k,\Hbarhat)$,
hence we obtain a canonical isomorphism $\Be_{S,\emptyset}(X)\simeq\Sh^1_{S,\emptyset}(k,\Hbarhat)$
and a canonical isomorphism $\Be_\omega(X)\simeq\Sh^1_\omega(k,\Hbarhat)$.
In  case (i) by Corollary \ref{cor:4cases} we have $\Sh^1_{S,\emptyset}(k,\Hbarhat)=0$,
hence $\Be_{S,\emptyset}(X)=0$.
In  case (ii) by  Lemma \ref{lem:S-1.3} we have $\Sh^1_\omega(k,\Hbarhat)=0$,
hence $\Be_\omega(X)=0$.
\qed
\end{subsec}


\section{Principal homogeneous spaces of connected groups}\label{sec:PHS}

In this section we prove Theorem \ref{thm:WA-S-pi1}.

\begin{subsec}\label{subsec:M}
Let $M$ be a $\Gal(\kbar/k)$-module, finitely generated over $\ZZ$.
Choose a $\ZZ$-free resolution
\begin{equation}\label{eq:resolution}
0\to L\to P\to M\to 0,
\end{equation}
where $L$ and $P$ are finitely generated $\ZZ$-free Galois modules.
We write
$$
\H^i(k,M^D):=\H^i(k,P^\vee\to L^\vee),
$$
where $P^\vee:=\Hom_\ZZ(P,\ZZ)$  is in degree 0 and $L^\vee:=\Hom_\ZZ(L,\ZZ)$ is in degree 1.
We regard $M^D:=(P^\vee\to L^\vee)$ as a dual complex to $M$.
Since the isomorphism class of $M^D$ in the derived category
does not depend on the choice of the resolution \eqref{eq:resolution},
the hypercohomology  $\H^i(k,M^D)$ also does not depend on the resolution.
\end{subsec}

\begin{lemma}\label{lem:4cases}
Let $M$ be as in \ref{subsec:M}.
Let $K/k$ be the smallest Galois extension in $\kbar$ splitting $M$.
Let $S\subset \sV$ be finite set of places of $k$.
   \par(i) If   any place $v\in S$ has a cyclic decomposition group in $\Gal(K/k)$,
then $\Sh^2_{S,\emptyset}(k,M^D)=0$.
   \par(ii) If  $K/k$ is a metacyclic extension,
then $\Sh^2_\omega(k,M^D)=0$.
\end{lemma}

\begin{proof}
Set $\gg=\Gal(K/k)$, then $M$ is a $\gg$-module.
We can choose  a resolution \eqref{eq:resolution}
such that $P$ is a permutation $\gg$-module and $L$ is a $\ZZ$-free $\gg$-module.
Then $P^\vee$ is a permutation module as well, and
by Lemma \ref{lem:Sha}  we have a canonical isomorphism
$$
\Sh^1_S(k,L^\vee)\isoto\Sh^2_S(k,P^\vee \to L^\vee)=\Sh^2_S(k,M^D),
$$
whence we obtain  canonical isomorphisms
\begin{gather*}
\Sh^1_{S,\emptyset}(k,L^\vee)\isoto\Sh^2_{S,\emptyset}(k,M^D),\\
\Sh^1_\omega(k,L^\vee)\isoto\Sh^2_\omega(k,M^D).
\end{gather*}
Since $K$ splits  $L^\vee$, in case (i) by Corollary \ref{cor:4cases} we have
$\Sh^1_{S,\emptyset}(k,L^\vee)=0$, hence
 $\Sh^2_{S,\emptyset}(k,M^D)=0$.
In  case (ii) by Lemma \ref{lem:S-1.3} we have $\Sh^1_\omega(k,L^\vee)=0$,
hence   $\Sh^2_\omega(k,M^D)=0$.
\end{proof}

\begin{subsec}{\em Proof of Theorem \ref{thm:WA-S-pi1}.}
By \cite[Lemma 6.8]{Sansuc} there is a canonical isomorphism $\Bra(X)\isoto\Bra(G)$.
By \cite[Cor.~7]{BvH}
there is a canonical isomorphism $\Bra(G)\isoto \H^2(k,\pi_1(G)^D)$.
Hence $\Be_S(X)\simeq \Sh^2_S(k,\pi(G)^D)$, whence
$\Be_{S,\emptyset}(X)\simeq \Sh^2_{S,\emptyset}(k,\pi_1(G)^D)$ and
$\Be_\omega(X)\simeq\Sh^2_\omega(k,\pi_1(G)^D)$.
In case (i) by Lemma \ref{lem:4cases}(i) we have $\Sh^2_{S,\emptyset}(k,\pi_1(G)^D)=0$,
hence $\Be_{S,\emptyset}(X)=0$.
In  case (ii) by Lemma \ref{lem:4cases}(ii) we have $\Sh^2_\omega(k,\pi_1(G)^D)=0$,
hence $\Be_\omega(X)=0$.
\qed
\end{subsec}

\section{Connected groups as homogeneous spaces of quasi-trivial groups}\label{sec:extension}

In this section we prove Proposition \ref{thm:extension}.

\begin{subsec}\label{subsec:ext}
{\em Proof of Proposition \ref{thm:extension}.}
We may and shall assume that $G$ is reductive, cf. proof of \cite[Prop.-Def. 3.1]{CT}.
Consider the biggest quotient torus $G\tor$ of $G$.
Set $M=\pi_1(G)$, then $\xx(G\tor)=M/M\tors$,
where $M\tors$ denotes the torsion subgroup of $M$.
Since $K$ splits $M$, we see that $K$ splits $\xx(G\tor)$.

We follow the construction in the proof of \cite[Prop.-Def. 3.1]{CT}.
Let $Z^0$ denote the radical (the identity component of the center) of our reductive group $G$.
Since $Z^0$ is isogeneous to $G\tor$, we see that $K$ splits $\xx(Z^0)$.
Set $\gg=\Gal(K/k)$, then $\xx(Z^0)$ is a $\gg$-module.
Choose a surjective homomorphism of $\gg$-modules $P\onto \xx(Z^0)$, where $P$ is a finitely generated permutation $\gg$-module.
We regard $P$ as a $\Gal(\kbar/k)$-module, then $K$ splits $P$.
Let $Q$ be the quasi-trivial $k$-torus with $\xx(Q)=P$.
We have a surjective homomorphism $\theta\colon Q\onto Z^0$.

Consider the canonical homomorphism
$$
\rho\colon G\sc\onto G\sss\into G.
$$
Set $G'=G\sc\times_k Q$, then  $G'$ is a quasi-trivial group and $K$ splits $\widehat{G'}=P^\vee$.
We define a surjective homomorphism
$$
\alpha\colon G'\to G,\quad \alpha(g,q)=\rho(g)\theta(q), \text{ where }g\in G\sc,\ q\in Q.
$$
Set $H=\ker\,\alpha$.
Note that 
and that $H$ is a central subgroup of multiplicative type of $G'$.
We have an exact sequence
\begin{equation}\label{eq:exact-H-G'-G}
1\to H\to G'\to G\to 1.
\end{equation}
Set $M':=\pi_1(G')=P$, then $K$ splits $M$ and $M'$.
By Lemma \ref{lem:Cyril} below
we have a canonical isomorphism of Galois modules
$\Hhat\cong \Ext^0_\ZZ(M'\to M,\ \ZZ)$ (where $M'$ is in degree 0).
It follows that $K$ splits $\Hhat$, which proves the proposition.
\qed
\end{subsec}

\begin{lemma}\label{lem:Cyril}
Assume we have a short exact sequence of Galois modules
$$
1\to H\to G'\labelto{\varphi} G\to 1,
$$
where $G$ and $G'$ are connected reductive $k$-groups over a field $k$ of characteristic 0,
and $H\subset G'$ is a central $k$-subgroup of multiplicative type.
Set $M:=\pi_1(G),\ M':=\pi_1(G')$.
Then there is a canonical isomorphism
$$
\XX(H)\cong\Ext^0_\ZZ(M'\to M,\ \ZZ),
$$
where in the complex $M'\to M$, the Galois module $M'$ is in degree 0 and $M$ is in degree 1,
and we write $\XX(H):=\Hhat$.
\end{lemma}

\begin{proof}[Proof\quad  {\rm (C.~Demarche).}]
Note that the homomorphism $G'\to G$ induces an isomorphism $G\psc\to G\sc$.
Choose compatible maximal tori $T_G\subset G$, $T_{G'}\subset G'$, $T_{G\sc}\subset G\sc$ and $T_{G\psc}\subset G\psc$.
It follows from the definition of $M$ and $M'$ that we have a commutative diagram with exact rows
\begin{equation}\label{eq:Cyril}
\xymatrix{
0\ar[r] &\xx(T_{G\psc})\ar[r]\ar[d]^\cong  &\xx(T_{G'})\ar[r]\ar[d] &M'\ar[r]\ar[d] &0  \\
0\ar[r] &\xx(T_{G\sc})\ar[r]               &\xx(T_{G})\ar[r]        &M \ar[r]       &0.
}
\end{equation}
Since the homomorphism $\varphi_*\colon G\psc\to G\sc$ is an isomorphism,
the left-hand vertical arrow $\xx(T_{G\psc})\to \xx(T_{G\sc})$ in diagram \eqref{eq:Cyril} is an isomorphism,
and the five lemma shows that the morphism of complexes of Galois modules
$$
(\xx(T_{G'})\to \xx(T_{G}))\longrightarrow (M'\to M)
$$
given by this diagram is a quasi-isomorphism.

The short exact sequence of complexes
$$
0\to (0\to \xx(T_{G}))\to (\xx(T_{G'})\to \xx(T_{G}))\to (\xx(T_{G'})\to 0)\to 0
$$
induces an exact sequence of Ext-groups
$$
\Hom_\ZZ(\xx(T_{G}),\ZZ)\to \Hom_\ZZ(\xx(T_{G'}),\ZZ)\to \Ext^0_\ZZ(\xx(T_{G'})\to \xx(T_{G}),\ \ZZ)\to \Ext^1_\ZZ(\xx(T_{G}),\ZZ)=0.
$$
Since the complex $\xx(T_{G'})\to \xx(T_{G})$ is quasi-isomorphic to $M'\to M$, we obtain an exact sequence
$$
\Hom_\ZZ(\xx(T_{G}),\ZZ)\to \Hom_\ZZ(\xx(T_{G'}),\ZZ)\to \Ext^0_\ZZ(M'\to M,\ \ZZ)\to 0,
$$
which we can write as
\begin{equation}\label{eq:exact-M}
\XX(T_{G})\to \XX(T_{G'})\to \Ext^0_\ZZ(M'\to M,\ \ZZ)\to 0,
\end{equation}
where  $\XX(T_{G})=\widehat{T_G}$ and $\XX(T_{G'})=\widehat{T_{G'}}$.
On the other hand, the exact sequence of $k$-groups of multiplicative type
$$
1\to H\to T_{G'}\to T_G\to 1
$$
gives an exact sequence
\begin{equation}\label{eq:exact-H}
0\to \XX(T_{G})\to \XX(T_{G'})\to \XX(H)\to 0.
\end{equation}
Comparing exact sequences \eqref{eq:exact-M} and \eqref{eq:exact-H},
we obtain a canonical isomorphism of Galois modules
\[
\XX(H)\cong \Ext^0_\ZZ(M'\to M,\ \ZZ). \qedhere
\]
\end{proof}

\begin{remark}
Constructing and arguing as in the proof of  \cite[Prop.-Def. 3.1]{CT},
we can construct an exact sequence \eqref{eq:exact-H-G'-G} with $G'$  a quasi-trivial $k$-group and
$H$ a flasque $k$-torus (and not just some $k$-group of multiplicative type)
such that the smallest Galois extension  $K/k$ in $\kbar$ splitting $\pi_1(G)$ splits both tori $G^{\prime\text{\,tor}}$ and $H$.
This strengthens Remark 3.1.1 of \cite{CT}.
\end{remark}

\section{Principal homogeneous spaces of connected groups again}\label{sec:alternative}

In this section we give an alternative proof of Theorem  \ref{thm:WA-S-pi1}
based on Proposition  \ref{thm:extension} and Theorem  \ref{thm:WA-S-H}.

\begin{subsec}\label{subsec:Galois-action}
Let $X$ be a homogeneous space of a  quasi-trivial $k$-group $G$ over a number field $k$.
Let $\Hbar\subset\Gbar$ be the stabilizer of a $\kbar$-point $\xbar\in X(\kbar)$.
We describe the action of $\Gal(\kbar/k)$ on $\Hbarhat$ defined by the homogeneous space $X$.

Let $h \in \Hbar(\kbar)$, then
$\xbar.h=\xbar$\,.
Let $\sigma\in \Gal(\kbar/k)$, then
${}^\sigma x. {}^\sigma h={}^\sigma x$\,.
For any $\sigma\in \Gal(\kbar/k)$ we
choose $g_\sigma\in G(\kbar)$ such that ${}^\sigma x =x.g_\sigma$
and the function $\sigma\mapsto g_\sigma$ is locally constant,
then
$$
g_\sigma\cdot {}^\sigma h\cdot g_\sigma^{-1}\in H(\kbar).
$$
The map $h\mapsto g_\sigma\cdot {}^\sigma h\cdot g_\sigma^{-1}$
comes from some $\sigma$-semialgebraic automorphism  (see \cite[\S 1.1]{Borovoi-Duke} for a definition) $\nu_\sigma$ of $\Hbar$,
which induces an automorphism $\widehat{\nu_\sigma}$ of $\Hbarhat$
(namely, $\widehat{\nu_\sigma}(\chi)(h)=\chi(\nu_\sigma^{-1}(h))$ for $\chi\in\Hbarhat$ and $h\in\Hbar(\kbar)$).
If we choose another element $g'_\sigma\in G(\kbar)$
such that ${}^\sigma x =x.g'_\sigma$, then
$g'_\sigma=h'g_\sigma$ for some $h'\in\Hbar(\kbar)$.
Then we obtain $\nu'_\sigma=\Inn(h')\circ\nu_\sigma$,
where $\Inn(h')$ is the inner automorphism of $\Hbar$ defined by $h'$.
We have $\widehat{\nu'_\sigma}=\widehat{\nu_\sigma}$,
because $\Inn(h')$  acts trivially on $\Hbarhat$.
The well-defined map $\sigma\mapsto \widehat{\nu_\sigma}$ is a homomorphism
defining an action of $\Gal(\kbar/k)$ on $\Hbarhat$.
\end{subsec}

\begin{subsec}\label{subsec:alternative}
{\em Alternative proof of Theorem \ref{thm:WA-S-pi1}.}
We deduce Theorem \ref{thm:WA-S-pi1} from Proposition \ref{thm:extension} and Theorem \ref{thm:WA-S-H}.
Since $K$ splits $\pi_1(G)$, by Proposition \ref{thm:extension}
we can write $G=G'/H$, where $G'$ is a quasi-trivial $k$-group
and $H$ is a central $k$-subgroup of multiplicative type in $G'$
such that $K$ splits $\Hhat$.

The group $G'$ acts on $X$ via $G$.
Let $\xbar\in X(\kbar)$, then the stabilizer of $\xbar$  in $\overline{G'}$ is $\Hbar:=H_\kbar$.
Consider the action $\sigma\mapsto \widehat{\nu_\sigma}$
of $\Gal(\kbar/k)$ on $\widehat{\Hbar}$ defined in \ref{subsec:Galois-action}.
Write ${}^\sigma x = x.g'_\sigma$, where $g'_\sigma\in G'(\kbar)$,
then
$$
\nu_\sigma(h)= g'_\sigma\cdot {}^\sigma h\cdot (g'_\sigma)^{-1}={}^\sigma h
$$
for $h\in H(\kbar)$,
because $H$ is central in $G'$.
It follows that the action of $\Gal(\kbar/k)$ on $\Hbarhat$ defined by the homogeneous space $X$ coincides
with the action on $\Hhat$ defined by the $k$-structure of $H$.

Now, since $K$ splits $\Hhat$, we see that $K$ splits $\Hbarhat$,
and Theorem \ref{thm:WA-S-pi1} follows from Theorem  \ref{thm:WA-S-H}.
\qed
\end{subsec}

\bigskip

\noindent\emph{Acknowledgements.}
We thank J.-L. Colliot-Th\'el\`ene for helpful e-mail correspondence,
 C. Demarche for proving Lemma \ref{lem:Cyril},
and the anonymous referee for a quick and thorough review.
The author finished this paper while visiting the Max-Planck-Institut f\"ur Mathematik, Bonn;
he thanks the Institute for hospitality, support, and excellent working conditions.



\begin{thebibliography}{BvH2}

\bibitem[AW]{AW}
M.\,F.\@ Atiyah and C.\,T.\,C.\@ Wall, \emph{Cohomology of groups},
  {A}lgebraic {N}umber {T}heory (J.W.S.\@ Cassels and A.\@ Fr{\"o}hlich, eds.),
  Academic Press/Thompson Book Co., Inc., London/Washington, D.C., 1967,
  pp.~94--115.

\bibitem[B1]{Bor90}
M.~Borovoi,
\emph{On weak approximation in homogeneous
spaces of simply connected algebraic  groups,}
in:  Automorphic functions and their applications (Khabarovsk, 1988),  64--81,
Acad. Sci. USSR, Inst. Appl. Math., Khabarovsk, 1990.

\bibitem[B2]{Borovoi-Duke}
M.~Borovoi,
\emph{Abelianization of the second nonabelian Galois cohomology,}
Duke Math. J.  {\bf 72}  (1993), 217--239.

\bibitem[B3]{Bor96}
M.~Borovoi,
\emph{The Brauer-Manin obstructions for homogeneous spaces
with connected or abelian stabilizer,}
J. reine angew. Math. {\bf 473} (1996), 181--194.

\bibitem[B4]{Bor98}
M.~Borovoi,
\emph{Abelian Galois cohomology of reductive groups,}
Mem. Amer. Math. Soc. {\bf 132} (1998), no.~626.

\bibitem[BvH1]{BvH}
 M.~Borovoi and J.~van Hamel,
\emph{Extended Picard complexes and linear algebraic groups},
J. reine angew. Math. {\bf 627} (2009), 53-82.

\bibitem[BvH2]{BvH2}
M.~Borovoi and J.~van Hamel,
\emph{Extended equivariant Picard complexes and homogeneous spaces},
Preprint, arXiv:1010.3414v2[math.AG].

\bibitem[CT]{CT}
J.-L.~Colliot-Th\'el\`ene,
\emph{R\'esolutions flasques des groupes lin\'eaires connexes},
J. reine angew. Math. {\bf 618} (2008), 77--133.

\bibitem[CTS]{CTS}
J.-L. Colliot-Th\'el\`ene et J.-J. Sansuc,
{\em La descente sur une vari\'et\'e
rationnelle d\'efinie sur un corps de nombres},
C. R. A. S. Paris {\bf 284} (1977), A1215--1218.

\bibitem[M]{Manin}
Yu.\,I.~Manin, \emph{Le groupe de Brauer-Grothendieck en g\'{e}om\'{e}trie diophantienne}, in:
Actes du Congr\`{e}s Intern. Math. (Nice, 1970), Tome 1,  Gauthier-Villars, Paris, 1971, pp. 401--411.


\bibitem[Sa]{Sansuc}
J.-J. Sansuc, \emph{Groupe de Brauer et arithm\'etique des groupes
alg\'ebriques lin\'eaires sur un corps de nombres}, J. reine angew. Math.
\textbf{327} (1981), 12--80.

\bibitem[Se]{Serre}
J.-P. Serre, \emph{Cohomologie galoisienne},
Lecture Notes in Math. \textbf{5}, 5\`eme ed., Springer-Verlag, Berlin 1994.

\bibitem[Sk]{Skorobogatov}
A.\,N. Skorobogatov,  {\it Torsors and Rational Points},
Cambridge Tracts in Mathematics {\bf 144},
Cambridge University Press, Cambridge 2001.





\end{thebibliography}
\end{document}